\documentclass[12pt]{amsart}
\usepackage{amsmath}
\usepackage{amssymb}
\usepackage{amscd}
\usepackage{multicol}
\usepackage{MnSymbol}
\usepackage{tikz}
\usepackage[OT2,T1]{fontenc}
\usepackage{mathrsfs}
\usepackage{comment} 
\usepackage{url}
\usepackage{stmaryrd}
\usepackage[all]{xy}
\usepackage{longtable}
\usepackage{multirow,bigdelim}
\DeclareSymbolFont{cyrletters}{OT2}{wncyr}{m}{n}

\DeclareMathOperator{\Mat}{Mat}

 \DeclareMathOperator{\wt}{wt}
\DeclareMathOperator{\Ext}{Ext}  \DeclareMathOperator{\SL}{SL}
\DeclareMathOperator{\Li}{Li}    
\DeclareMathOperator{\rank}{rank}

\DeclareMathOperator{\dep}{dep}

\title[On lower bounds of the dimensions of multizeta values]{O\MakeLowercase{n lower bounds of the dimensions of multizeta values in positive characteristic}}

\author{Yen-Tsung Chen and Ryotaro Harada}

\address{Department of Mathematics, National Tsing Hua University, Hsinchu City 30042, Taiwan R.O.C.}

\email{ytchen.math@gmail.com}

\address{National Center for Theoretical Sciences, No.101, Kuang Fu Road, Sec. 2, Hsinchu City 30043 Taiwan R.O.C}
\email{ryotaro@ncts.ntu.edu.tw}
\date{December 1, 2020}

\newtheorem{thm}{Theorem}[section]
\newtheorem{lem}[thm]{Lemma}
\newtheorem{cor}[thm]{Corollary}
\newtheorem{prop}[thm]{Proposition}
\theoremstyle{remark}
        
\theoremstyle{definition}
\newtheorem{defn}[thm]{Definition}
\newtheorem{rem}[thm]{Remark}

\newtheorem{eg}[thm]{Example}

\numberwithin{equation}{section}
\begin{document}
\bibliographystyle{amsalpha+}

\begin{abstract}
In this paper, we study the linear independence of special values, including the positive characteristic analogue of multizeta values, alternating multizeta values and multiple polylogarithms, at algebraic points.
Consequently, we establish linearly independent sets of these values with the same weight indices and a lower bound on the dimension of the space generated by depth $r>2$ multizeta values of the same weight in positive characteristic.
\end{abstract}

\maketitle
\tableofcontents
\setcounter{section}{0}

\section{Introduction}

\subsection{Multizeta values in positive characteristic}
    Let $\mathbb{F}_q$ be a fixed finite field with $q$ elements, where $q$ is the power of a prime number $p$. 
    Let $\mathbb{P}^1$ be the projective line defined over $\mathbb{F}_q$ with a fixed point at infinity $\infty\in\mathbb{P}^1(\mathbb{F}_q)$. 
    Let $A$ be the ring of regular functions away from $\infty$, and let $k$ be its fraction field. Let $k_\infty$ be the completion of $k$ at $\infty$, and let $\mathbb{C}_\infty$ be the completion of a fixed algebraic closure of $k_\infty$. Let $\theta$ be a variable. We identify $A$ with the polynomial ring $\mathbb{F}_q[\theta]$ and $k$ with the rational function field $\mathbb{F}_q(\theta)$.
    
    The $\infty$-adic multizeta values ($\infty$-adic MZVs) 
    were defined by Thakur in \cite{Th04} as a generalization of Carlitz zeta values \cite{Ca35}. For any index $\mathfrak{s}=(s_1,\dots,s_r)\in\mathbb{Z}_{>0}^r$, $\infty$-adic MZVs are defined by the following series:
    \begin{equation}\label{Eq:Definition_of_MZVs}
        \zeta_A(\mathfrak{s}):=\sum\frac{1}{a_1^{s_1}\cdots a_r^{s_r}}\in k_\infty,
    \end{equation}
where the sum is over $(a_1,\dots,a_r)\in A^r$ with $a_i$ monic and $\deg a_1> \deg a_2>\cdots >\deg a_r$. The second author introduced and studied  $\infty$-adic alternating multizeta values ($\infty$-adic AMZVs) \cite{H20}, which are a generalization of $\infty$-adic MZVs defined by the following series:
    \begin{equation}\label{Eq:Definition_of_AMZVs}
        \zeta_A(\mathfrak{s};\boldsymbol{\epsilon}):=\sum\frac{\epsilon_1^{\deg a_1}\cdots\epsilon_r^{\deg a_r}}{a_1^{s_1}\cdots a_r^{s_r}}\in k_\infty,
    \end{equation}
where $\boldsymbol{\epsilon}:=(\epsilon_1, \ldots, \epsilon_r) \in \left(A^\times\right)^r=\left(\mathbb{F}_q^\times\right)^r$ and the sum is over $(a_1,\dots,a_r)\in A^r$ with $a_i$ monic and $\deg a_1> \deg a_2>\cdots >\deg a_r$. The weight and depth of the presentation $\zeta_A(\mathfrak{s})$ and $\zeta_A(\mathfrak{s};\boldsymbol{\epsilon})$ are defined by $\wt(\mathfrak{s}):=s_1+\cdots+s_r$ and $\dep(\mathfrak{s}):=r$, respectively. Note that both $\infty$-adic MZVs and AMZVs are introduced as positive characteristic counterparts of real-valued multizeta values (real-valued MZVs) and real-valued alternating multizeta values (real-valued AMZVs), which researchers have conducted a variety of interesting studies (for details, see \cite{Zh16}).

    The $\infty$-adic MZVs (resp. $\infty$-adic AMZVs) are non-vanishing according to \cite{Th09a} (resp. \cite{H20}). Additionally, both $\infty$-adic MZVs \cite{AT09} and $\infty$-adic AMZVs \cite{H20} appear as periods of certain pre-$t$-motives introduced in \cite{P08}. Let $\mathcal{Z}$ (resp. $\mathcal{AZ}$) be the $k$-vector space spanned by $1$ and all $\infty$-adic MZVs (resp. $\infty$-adic AMZVs). For $w\geq 1$, let $\mathcal{Z}_w$ (resp. $\mathcal{AZ}_w$) be the $k$-vector space spanned by $\infty$-adic MZVs (resp. $\infty$-adic AMZVs) of weight $w$. In \cite{Th10} (resp. \cite{H20}), it was shown that the product of two $\infty$-adic MZVs (resp. $\infty$-adic AMZVs) can be written as an $\mathbb{F}_p$-linear combination of $\infty$-adic MZVs (resp. $\infty$-adic AMZVs) of the same weight, where $\mathbb{F}_p$ is the prime field of $k$. Thus, $\mathcal{Z}$ (resp. $\mathcal{AZ}$) forms a $k$-algebra. Specifically, one has $\mathcal{Z}_{w_1}\mathcal{Z}_{w_2}\subset\mathcal{Z}_{w_1+w_2}$ (resp. $\mathcal{AZ}_{w_1}\mathcal{AZ}_{w_2}\subset\mathcal{AZ}_{w_1+w_2}$) for $w_1\geq 1$ and $w_2\geq 1$. Further, an analogue of Goncharov's direct sum conjecture \cite{G97} for $\infty$-adic MZVs (resp. $\infty$-adic AMZVs) was established in \cite{C14} (resp. \cite{H20}), namely, $\mathcal{Z}$ (resp. $\mathcal{AZ}$) forms a graded $k$-algebra (graded by weights). In other words, all $k$-linear relations among $\infty$-adic MZVs (resp. $\infty$-adic AMZVs) are generated by these $k$-linear relations among $\infty$-adic MZVs (resp. $\infty$-adic AMZVs) with the same weight. 
    
    To study $k$-linear relations among the $\infty$-adic MZVs with the same weight, Todd \cite{To18} used the power sum and lattice reduction methods to produce $k$-linear relations. Based on his result, an analogue of Zagier's dimension conjecture, which predicts the dimension of the $\mathbb{Q}$-vector space generated by the same weight real-valued MZVs over $\mathbb{Q}$ (see \cite{W12}), has been formulated in the positive characteristic setting.
Although some $k$-linear relations among the $\infty$-adic MZVs with the same weight have been discovered (see \cite{LRT14}, \cite{C16} \cite{Ch17}, \cite{To18}, \cite{CPY19} and \cite{GP20}),  not much is known about the $k$-linear independence of $\infty$-adic MZVs with the same-weight \footnote{In Remark \ref{Rem:Alg_Ind}, we mention the algebraic independence results of some $\infty$-adic MZVs obtained by \cite{CY07} and \cite{Mi15a, Mi17}.}. 
    Let $r,~w\in\mathbb{Z}_{>0}$ and 
    $$\mathcal{Z}_w^{r}:=\mathrm{Span}_k\{\zeta_A(\mathfrak{s})\mid\wt(\mathfrak{s})=w,~\dep(\mathfrak{s})=r\}$$
    be the $k$-vector space spanned by $\infty$-adic MZVs of weight $w$ and depth $r$. Also, we set $\mathcal{Z}_w^{1,r}:=\mathcal{Z}_w^{1}+\mathcal{Z}_w^{r}$. 
    In \cite{C16}, Chang proved a necessary condition for the linear dependence of depth 2 $\infty$-adic MZVs, which can be stated as follows.
    \begin{thm}[\cite{C16}]\label{Thm:Chang_16}
        Let $w\in\mathbb{Z}_{>0}$ with $w\geq 2$. For each $i=1,\dots,m$, let $\mathfrak{s}_i:=(s_{i1},s_{i2})\in\mathbb{Z}_{>0}^2$ be chosen with $s_{i1}+s_{i2}=w$. Then, all $k$-linear relations among $\{\zeta_A(w),\zeta_A(\mathfrak{s}_1),\dots, \zeta_A(\mathfrak{s}_m)\}$ are those coming from the $k$-linear relations among $\{\zeta_A(w)\}\cup\{\zeta_A(\mathfrak{s}_j)\ |\ \ s_{j2}\ \text{is divisible by}\ q-1 \}$. Specifically, we have
        $$\dim_k\mathcal{Z}_w^{1,2}\geq w-\lfloor\frac{w-1}{q-1}\rfloor$$
        where $\lfloor -\rfloor$ is the floor function. In particular,
        $$\dim_k\mathcal{Z}_w^2\geq w-1-\lfloor\frac{w-1}{q-1}\rfloor.$$
    \end{thm}

    One of the goals of the present paper is to study the generalization of Theorem~\ref{Thm:Chang_16} in the higher depth case.

\subsection{Statement of the main theorem}
    Set $L_0:=1$ and $L_i:=(\theta-\theta^q)\cdots(\theta-\theta^{q^i})$ for each $i\geq 1$. For an $r$-tuple $\mathfrak{s}=(s_{1},\dots,s_{r})\in \mathbb{Z}_{>0}^{r}$, the Carlitz multiple polylogarithms (CMPLs) are defined as follows (see~\cite{C14}):
    \begin{equation}\label{Eq:CMPL}
        \Li_\mathfrak{s}(z_1,\dots,z_r):=\underset{i_1>\cdots>i_r\geq 0}{\sum}\frac{z_1^{q^{i_1}}\cdots z_r^{q^{i_r}}}{L_{i_1}^{s_1}\cdots L_{i_r}^{s_r}}\in k\llbracket z_1,\ldots,z_r \rrbracket.
    \end{equation}
    
    CMPLs can be viewed as an analogue of classical multiple polylogarithms 
    $$\mathscr{L}_\mathfrak{s}(z_1,\dots,z_r):=\sum_{n_1>\cdots>n_r>0}\frac{z_1^{n_1}\cdots z_r^{n_r}}{n_1^{s_1}\cdots n_r^{s_r}}\in\mathbb{Q}\llbracket z_1,\dots,z_r\rrbracket$$
    in the positive characteristic setting.
The specializations at $(1,\dots,1)$ of classical multiple polylogarithms with several variables give the real-valued MZVs. This phenomenon becomes delicate in the positive characteristic setting.

In \cite{AT90}, it is shown that depth 1 MZVs can be written as $k$-linear combinations of same-weight Carlitz polylogarithms (CMPLs with $r=1$) at algebraic points. This is generalized to higher depth MZVs case in \cite{C14}.
Additionally, some interesting algebraic relations between $\infty$-adic MZVs and Carlitz logarithms at algebraic points were discovered by Thakur in \cite[Thm.~6]{Th09b}. Using the stuffle relations of CMPLs (see \cite[Sec.~5.2]{C14}), we deduce some $k$-linear relations between $\infty$-adic MZVs and CMPLs of the same weight at algebraic points. With these relations in mind, one may naturally ask whether we can extend the linear independence results, such as Theorem~\ref{Thm:Chang_16}, to special values in the higher depth case, including $\infty$-adic MZVs, $\infty$-adic AMZVs, and CMPLs at algebraic points.

    In what follows, we first introduce some terminology used to formulate our main result. For a fixed $w\in\mathbb{Z}_{>0}$, let
    \begin{equation}\label{Eq:Index_1}
        I(w):=\{\mathfrak{s}=(s_1,\dots,s_r)\in\mathbb{Z}_{>0}^r\mid\wt(\mathfrak{s})=w,~1\leq r\leq w\}
    \end{equation}
    be the collection of all weight $w$ indices and let
    \begin{equation}\label{Eq:Index_2}
        J(w):=\{T\subset\{1,\dots,w-1\}\mid T\neq\emptyset\}\cup\{\{0\}\}.
    \end{equation}
    
    Let $\mathfrak{s}_i=(s_{i1},\dots,s_{ir})\in\mathbb{Z}_{>0}^r$ and $\mathfrak{Q}_i:=(Q_{i1},\dots,Q_{ir})\in \overline{k}[t]^r$ for $i\in\mathbb{Z}_{\geq 0}$. There is a family of power series $\mathscr{L}_{[i],j}$ in the variable $t$ 
        with coefficients in $\overline{k}$ and depending on datum $(\mathfrak{s}_i,\mathfrak{Q}_i)$ that defines entire functions on $\mathbb{C}_\infty$ such that the specialization of these series at $t=\theta$ recovers $\infty$-adic MZVs, $\infty$-adic AMZVs, and CMPLs at algebraic points. More precisely, for $n\in\mathbb{Z}$ we define the $n$-fold Frobenius twisting
    \begin{align*}
	    \mathbb{C}_{\infty}((t))&\rightarrow\mathbb{C}_{\infty}((t))\\
	    f:=\sum_{i}a_it^i&\mapsto \sum_{i}a_i^{q^{n}}t^i=:f^{(n)}
    \end{align*}
    and consider the entire power series
    \begin{equation}\label{Eq:Anderson_Thakur_Series}
        \Omega(t):=(-\theta)^{\frac{-q}{q-1}}\prod_{i=1}^\infty\left(1-\frac{t}{\theta^{q^i}}\right)\in\overline{k}\llbracket t\rrbracket
    \end{equation}
    where $(-\theta)^{\frac{1}{q-1}}$ is a fixed $(q-1)$-th root of $-\theta$ such that $1/\Omega(\theta)=\Tilde{\pi}$ (See ~\cite[Cor.~5.1.4]{ABP04}) and $\Tilde{\pi}$ is the Carlitz period.
    Then 
    \begin{equation}\label{Eq:Deformation_Series}
        \mathscr{L}_{[i],j}:=\sum_{\ell_1>\cdots>\ell_{j}\geq 0}
        (\Omega^{s_{ij}}Q_{ij})^{(\ell_j)}\cdots(\Omega^{s_{i1}}Q_{i1})^{(\ell_1)}\in\overline{k}\llbracket t\rrbracket.
    \end{equation}
    
    For example, the \emph{Carlitz logarithm} is the CMPL for $\mathfrak{s}=(1)$, which is given by
    \begin{equation}\label{Eq:Carlitz_Log}
        \log_C(z)=\Li_{(1)}(z)=\sum_{i\geq 0}\frac{z^{q^i}}{L_i}\in k\llbracket z\rrbracket.
    \end{equation}
    Carlitz first noticed that $\log_C(1)=\zeta_A(1)$ in \cite{Ca35}.
    
    Let $u\in\overline{k}$ be such that $\log_C(z)$ converges at $z=u$. Then, we consider the power series (see \cite[Sec.~6.1.1]{P08}):
    $$\mathcal{L}_u(t):=u+\sum_{i>0}\frac{u^{q^i}}{(t-\theta^q)\cdots(t-\theta^{q^i})}\in\overline{k}\llbracket t\rrbracket.$$
    It can be shown that $\mathcal{L}_u(t)$ converges at $t=\theta$ and that the evaluation recovers the Carlitz logarithm at $z=u$, namely,
    $$\mathcal{L}_u(t)\mid_{t=\theta}=\log_C(z)\mid_{z=u}.$$
    One may see from the definition that the product $\Omega(t)\mathcal{L}_u(t)$ provides an example of the special series $\mathscr{L}_{[i],j}$ (see Proposition \ref{Prop:Specialization} for further details).
     
    We characterize our linear independence criterion from the information of the $r$-tuple $\mathfrak{s}\in\mathbb{Z}_{> 0}^r$ of these special values. Therefore, the $r$-tuple $\mathfrak{s}$ reflects the vanishing order of the corresponding family of the power series at $t=\theta^{q^i}$ with $i\in\mathbb{Z}_{>0}$. This fact allows us to introduce a certain notion for the $r$-tuple $\mathfrak{s}$. Specifically, we define the following map:
    \begin{align*}\label{Eq:g_Independent}
        g:I(w)&\to J(w)\\
        (s_1,\dots,s_r)&\mapsto\{w-s_1,w-s_1-s_2,\dots,w-s_1-\cdots-s_{r-1}\},~\mbox{for}~r\geq 2\\
        (w)&\mapsto\{0\}.
    \end{align*}
    
    Then, for each collection of weight $w$ indices $S\subset I(w)$, we say $S$ is $g$-independent if $g(\mathfrak{s})\cap g(\mathfrak{s}')=\emptyset$ for any $\mathfrak{s},~\mathfrak{s}'\in S$ with $\mathfrak{s}\neq\mathfrak{s}'$. In fact, we will see in Lemma~\ref{Eq:Vanishing_Order_of_psi''} that $g(\mathfrak{s})$ is a collection of the vanishing orders of certain entire series at $t=\theta^{q^i}$ with $i\in\mathbb{Z}_{>0}$. Note that the difference in the vanishing orders implies the $\overline{k}(t)$-linear independence of these entire series (see Lemma~\ref{Lem:Vanishing_of_Entire_Functions}). Then, an application of \cite[Thm.~3.1.1]{ABP04} provides the desired $\overline{k}$-linear independence of these special values. Now, we state our main theorem as follows; it will later be restated as Theorem~\ref{Thm:Main_Thm}.
    
    \begin{thm}\label{Thm:Main_Thm_Intro}
        For $w\in\mathbb{Z}_{>0}$, let $S=\{\mathfrak{s}_0,\dots,  \mathfrak{s}_m\}\subset I(w)$ be $g$-independent with $\mathfrak{s}_0=(w)$. Let $\mathfrak{Q}_i\in\overline{k}[t]^{\dep(\mathfrak{s}_i)}$, which satisfies conditions (\ref{Eq:Condition_of_Q_1}) and (\ref{Eq:Condition_of_Q_2})
               and then if we set $\mathscr{L}_{[i]}:=\mathscr{L}_{[i],\dep(\mathfrak{s}_i)}$ for $i=0, \dots, m$, the following set 
        \[
            \{ \mathscr{L}_{[0]}(\theta), \ldots,  \mathscr{L}_{[m]}(\theta)\}
        \]
        is $k$-linearly independent.
            \end{thm}
    
    For suitable choices of $\mathfrak{Q}_i$, the specialization of $\mathscr{L}_{[i]}$ at $t=\theta$ recovers MZVs, AMZVs, and CMPLs at algebraic points (see Proposition~\ref{Prop:Specialization}). Moreover, it is easily seen from the definition that $S_{\leq 2}^w:=\{w\}\cup\{\mathfrak{s}:=(s_1,s_2)\mid s_1+s_2=w,~s_2\text{ is not divisible by }q-1\}$ is $g$-independent. In the case of $S=S_{\leq 2}^w$ with suitable choices of $\mathfrak{Q}_i$, Theorem~\ref{Thm:Main_Thm_Intro} shows that 
    $$\{\zeta_A(w)\}\cup\{\zeta_A(\mathfrak{s}_j)\ |\ \ \mathfrak{s}_j=(s_{j1},s_{j2}),~s_{j1}+s_{j2}=w,~s_{j2}\ \text{is not divisible by}\ q-1 \}$$ 
    is a $k$-linearly independent set. In particular, Theorem~\ref{Thm:Main_Thm_Intro} recovers Chang's result of Theorem~\ref{Thm:Chang_16}.
    
    Additionally, Theorem~\ref{Thm:Main_Thm_Intro} provides many $k$-linearly independent $\infty$-adic MZVs of the same weight in the higher-depth case. As a consequence, we deduce the following corollary, which provides a lower bound on the dimension of $\mathcal{Z}_w^r$ with the given weight $w$ and depth $r\geq 2$. This will later be restated as Corollary~\ref{Cor:Lower_Bound}.
    
    \begin{cor}\label{Cor:Lower_Bound_Intro}
        For given $w,r\in\mathbb{Z}_{>0}$, we have
        $$\dim_k\mathcal{Z}_{w}^{1,r}\geq 1+ \lfloor\frac{w-1-\lfloor\frac{w-1}{q-1}\rfloor}{r-1}\rfloor$$
        where $\lfloor -\rfloor$ is the floor function. In particular,
        $$\dim_k\mathcal{Z}_{w}^r\geq \lfloor\frac{w-1-\lfloor\frac{w-1}{q-1}\rfloor}{r-1}\rfloor.$$
    \end{cor}
    
 In the case of $q\neq 2$, given an arbitrary $n,~r\in\mathbb{Z}_{>0}$, we can find $w:=w_{q, n, r}\in\mathbb{Z}_{>0}$, which depends on $q,~n$ and $r$, such that there are at least $n$ distinct $k$-linearly independent $\infty$-adic MZVs with the same depth $r$ and the same weight $w$. 

One motivation of our main theorem arises from the study of the $\mathbb{Q}$-vector space generated by real-valued MZVs. 
Zagier \cite{Z94} gave a conjecture on the dimension of the $\mathbb{Q}$-vector space spanned by depth $2$ real-valued MZVs of fixed weight $n\geq 3$ in terms of the dimension of cusp forms of weight $n$ for $\SL_2(\mathbb{Z})$.
This conjecture is partially solved in \cite{Z93} by giving the upper bound of the dimension, but the lower bound remains unknown. For the higher-depth case, Broadhurst and Kreimer proposed a conjecture in \cite{BK97} that predicts that the dimensions of the weight- and depth-graded parts of the $\mathbb{Q}$-vector space generated by real-valued MZVs are given in the specific generating series. In their conjecture, the generating series is explicitly described using the dimensions of the weight-graded parts of the $\mathbb{Q}$-vector space generated by cusp forms.    
The Broadhurst-Kreimer conjecture in the general depth case remains an open problem. However, Goncharov \cite{G98} proved that in the depth 3 case, the conjecture provides the upper bound of the dimension. Another proof was presented by Ihara and Ochiai \cite{IO08}.  
In the depth $\geq 4$ case, we do not even know if their conjecture provides the upper bound of the dimension.  

On the other hand, in the positive characteristic case, we can consider the dimension of the $k$-linear space generated by $\infty$-adic MZVs of fixed weight and depth. The depth 2 case has been studied by \cite{C16}, and  the lower bound of the dimension is given (see Theorem \ref{Thm:Chang_16}). The depth $\geq 2$ case is discussed in this paper, and the lower bound is given by Corollary \ref{Cor:Lower_Bound_Intro} as a consequence of Theorem \ref{Thm:Main_Thm_Intro}. Thus, our result does not provide the upper bound of the dimension, but we can obtain the lower bound of the dimension in the higher-depth case.


 We end this section by outlining this paper. In \S\ref{Notations}, we provide the notation of the basic objects used in our exposition.
In \S\ref{t-motive}, we review the definitions of Anderson dual $t$-motives and introduce a specific one constructed by means of the fiber coproduct, which was developed in \cite{CM20}. 
In \S\ref{t-module}, we present the definitions and results of Anderson $t$-modules and certain $\Ext^1$-modules. In \S\ref{AT-polynomial}, we revisit the definition of Anderson-Thakur polynomials and that the deformation series can be specialized to CMPLs, $\infty$-adic MZVs and $\infty$-adic AMZVs. We prove our main theorem in \S\ref{Main and app} based on the preliminaries in \S\ref{Pre} and conclude with applications that provide linear independence sets of special values and Corollary \ref{Cor:Lower_Bound_Intro}.

\section{Preliminaries}\label{Pre}
\subsection{Notations}\label{Notations}
\label{No}
We now define the following notation.

\begin{itemize}
\setlength{\leftskip}{1.0cm}
\item[$q=$] a power of a prime number $p$.  
\item[$\mathbb{F}_q=$] a finite field with $q$ elements.
\item[$\theta$, $t=$] independent variables.
\item[$A=$] the polynomial ring $\mathbb{F}_q[\theta]$.
\item[$A_{+}=$] the set of monic polynomials in $A$.
\item[$A_{d+}=$] the set of elements in $A_{+}$ of degree $d$. 
\item[$k=$] the rational function field $\mathbb{F}_q(\theta)$.
\item[$k_{\infty}=$] the completion of $k$ at the infinite place $\infty$, $\mathbb{F}_q((\frac{1}{\theta}))$.
\item[$\overline{k_{\infty}}=$] a fixed algebraic closure of $k_{\infty}$.
\item[$\mathbb{C}_{\infty}=$] the completion of $\overline{k_{\infty}}$ at the infinite place $\infty$.
\item[$\overline{k}=$] a fixed algebraic closure of $k$ in $\mathbb{C}_{\infty}$.
\item[$k^{\rm sep}=$] a fixed separable closure of $k$ in $\overline{k}$.
\item[$|\cdot|_{\infty}=$] a fixed absolute value for the completed field $\mathbb{C}_{\infty}$ such that $|\theta|_{\infty}=q$.
\item[$\mathbb{T}=$] the Tate algebra over $\mathbb{C}_{\infty}$, the subring of $\mathbb{C}_{\infty}\llbracket t \rrbracket$ consisting of\\
\quad power series convergent on the closed unit disc $|t|_{\infty}\leq 1$.
\item[$\mathcal{E}=$] $\{\sum_{i=0}^{\infty}a_it^i\in\overline{k}\llbracket t \rrbracket\mid \lim_{i\to\infty}|a_i|_\infty^{1/i}=0,~[k_\infty(a_0,a_1,\dots):k_\infty]<\infty\}.$
\item[$\mathrm{ord}_\alpha(f)=$] the vanishing order of $f\in\mathcal{E}$ at $\alpha\in\mathbb{C}_\infty$.
\item[$D_i=$] $\prod^{i-1}_{j=0}(\theta^{q^i}-\theta^{q^j})\in A_{+}$, where $D_0:=1$.
\item[$\Gamma_{n+1}=$] the Carlitz gamma, $\prod_{i}D_i^{n_i}$ ($n = \sum_{i}n_iq^i\in\mathbb{Z}_{\geq0} \ (0\leq n_i\leq q-1)$).
\end{itemize}

\subsection{Anderson dual $t$-motives and Frobenius modules}\label{t-motive}
In this section, we recall the notion of Frobenius modules and Anderson dual $t$-motives. Here, we use the term Anderson dual $t$-motives for those called dual $t$-motives in \cite[Def.~4.4.1]{ABP04} and Anderson $t$-motives in \cite[Def. 3.4.1]{P08}.

We denote by $\overline{k}[t, \sigma]$ the non-commutative $\overline{k}[t]$-algebra generated by $\sigma$ subject to the following relation:
\[
 \sigma f=f^{(-1)}\sigma, \quad f\in\overline{k}[t].
\]

\begin{defn}[{\cite[Sec.~2.2]{CPY19}}]
A {\it Frobenius module} is a left $\overline{k}[t, \sigma]$-module that is free of finite rank over $\overline{k}[t]$. We define morphisms of Frobenius modules by left $\overline{k}[t, \sigma]$-module homomorphisms and denote by $\mathscr{F}$ the category of Frobenius modules.
\end{defn}

Let $M$ be a Frobenius module with a fixed $\overline{k}[t]$-basis $\mathbf{m}=(\mathbf{m}_1,\dots,\mathbf{m}_r)^\mathrm{tr}$, where $r=\rank_{\overline{k}[t]}M$. Then, the $\sigma$-action can be represented by a matrix $\Phi\in\Mat_r(\overline{k}[t])$. In other words, we have $\sigma\mathbf{m}=\Phi\mathbf{m}$. Conversely, once we fix a free $\overline{k}[t]$-module $M$ of rank $r$ with a fixed $\overline{k}[t]$-basis $\mathbf{m}=(\mathbf{m}_1,\dots,\mathbf{m}_r)^\mathrm{tr}$ and $\Phi\in\Mat_r(\overline{k}[t])$, we can naturally obtain a Frobenius module structure on $M$ by setting $\sigma\mathbf{m}=\Phi\mathbf{m}$. In this case, we call $M$ the Frobenius module defined by $\Phi$.

\begin{defn}
An {\it Anderson dual $t$-motive} is a left  $\overline{k}[t, \sigma]$-module $M$ satisfying that
\begin{itemize}
    \item[(i)] $M$ is a free left $\overline{k}[t]$-module of finite rank,
    \item[(ii)] $M$ is a free left $\overline{k}[\sigma]$-module of finite rank,
    \item[(iii)] $(t-\theta)^s M\subset \sigma M$ for all sufficiently large $s\in\mathbb{Z}$.
\end{itemize}
According to the above definition, it is clear that an Anderson dual $t$-motive is also a Frobenius module.
\end{defn}

We give the following basic but important examples of a Frobenius module and an Anderson $t$-module.

\begin{eg}\label{Ex:Trivial_Frobenius_Module}
    The \emph{trivial Frobenius module} is defined by $\mathbf{1}=\overline{k}[t]$ with the $\sigma$-action  given by $\sigma f=f^{(-1)}$ for each $f\in\overline{k}[t]$. Note that $\mathbf{1}$ is \emph{not} an Anderson dual $t$-motive, as it is not of finite rank as a left $\overline{k}[\sigma]$-module.
\end{eg}

\begin{eg}\label{Ex:Tensor_Powers_of_Carlitz_Motive}
    Let $n\in\mathbb{Z}_{>0}$. Then, the {\it$n$-th tensor power of the Carlitz motive} is defined by $C^{\otimes n}=\overline{k}[t]$ with the $\sigma$-action  given by $\sigma f=f^{(-1)}(t-\theta)^n$ for each $f\in\overline{k}[t]$. Note that $C^{\otimes n}$ is an Anderson dual $t$-motive, the set $\{1\}$ forms a $\overline{k}[t]$-basis of $C^{\otimes n}$ and the set $\{(t-\theta)^{n-1},\dots,(t-\theta),1\}$ forms a $\overline{k}[\sigma]$-basis of $C^{\otimes n}$.
\end{eg}

Next, we construct the Anderson dual $t$-motive $M'_{\star}$ via the fiber coproduct method, which was introduced in \cite{CM20}. To begin, let us introduce some notions. We first recall the power series defined in (\ref{Eq:Anderson_Thakur_Series})
\begin{align*}
    \Omega(t)=(-\theta)^{\frac{-q}{q-1}}\prod_{i=1}^\infty\left(1-\frac{t}{\theta^{q^i}}\right)\in\mathcal{E}.
\end{align*}
Note that $\Omega(t)$ satisfies the Frobenius difference equation 
\begin{align}\label{domega}
\Omega^{(-1)}=(t-\theta)\Omega.
\end{align}

Given a polynomial $Q:=\sum_{i}c_it^i\in \overline{k}[t]$, we define $||Q||_\infty:=\max_i\{|c_i|_\infty\}$. For an $r$-tuple $\mathfrak{s}=(s_1,\dots,s_r)\in\mathbb{Z}_{>0}^r$, we set  $\mathbb{D}_{\mathfrak{s}}$ as the collection of all $\mathfrak{Q}:=(Q_1,\dots,Q_r)\in \overline{k}[t]^r$ satisfying the following two conditions:
\begin{equation}\label{Eq:Condition_of_Q_1}
        \left( ||Q_1||_\infty/|\theta|_\infty^{qs_1/(q-1)} \right)^{q^{i_1}}\cdots
        \left( ||Q_r||_\infty/|\theta|_\infty^{qs_r/(q-1)} \right)^{q^{i_r}}\to 0 \ \ \mbox{as}\ \ 0\leq i_1<\cdots<i_r\to\infty,
\end{equation}
\begin{equation}\label{Eq:Condition_of_Q_2}
    aQ_r\not\in(\sigma-1)C^{\otimes s_r}\mbox{ for all }a\in\mathbb{F}_q[t].
\end{equation}

    Condition (\ref{Eq:Condition_of_Q_1}) guarantees that the series $\mathscr{L}_{[i],j}$ given in (\ref{Eq:Deformation_Series}) defines an entire function, and condition (\ref{Eq:Condition_of_Q_2}) implies that the special point we constructed is not an $\mathbb{F}_q[t]$-torsion element under the isomorphism given in Theorem~\ref{Thm:Anderson_Ext}.
    
    Next, we set Anderson dual $t$-motives $\{M'_{[i]}\}$ and construct the fiber coproduct over $C^{\otimes w}$ of them. Let $w\in\mathbb{Z}_{>0}$,  $S=\{\mathfrak{s}_0,\dots,\mathfrak{s}_m\}\subset I(w)$ with $\mathfrak{s}_0=(w)$. 
    Let $r_i:=\mathrm{dep}(\mathfrak{s}_i)$ for each $0\leq i\leq m$. For $\mathfrak{s}_i=(s_{i1},\dots,s_{ir_i})\in I(w)$ and $\mathfrak{Q}_i=(Q_{i1},\dots,Q_{ir_i})\in\mathbb{D}_{\mathfrak{s}_i}$, we define the square matrix $\Phi_{[i]}\in\Mat_{r_i+1}(\overline{k}[t])$ and the column vector $\psi_{[i]}\in\Mat_{(r_i+1)\times 1}(\mathcal{E})$ as follows: 
\[
    \Phi_{[i]}:=
    \begin{pmatrix}
    (t-\theta)^{s_{i1}+\cdots+s_{ir_i}} & 0 & 0 & \cdots & 0 \\
    Q_{i1}^{(-1)}(t-\theta)^{s_{i1}+\cdots+s_{ir_i}} & (t-\theta)^{s_{i2}+\cdots+s_{ir_i}} & 0 & \cdots & 0\\
    0 & Q_{i2}^{(-1)}(t-\theta)^{s_{i2}+\cdots+s_{ir_i}} & \ddots &  & \vdots \\
    \vdots &  & \ddots & (t-\theta)^{s_{ir_i}} & 0\\
    0 & \cdots & 0 & Q_{ir_i}^{(-1)}(t-\theta)^{s_{ir_i}} & 1
    \end{pmatrix},
\]
\[
    \psi_{[i]}:=
        \left(
        \begin{array}{clr}
            \Omega^{s_{i1}+\cdots+s_{ir_i}}  \\
            \Omega^{s_{i2}+\cdots+s_{ir_i}}\mathscr{L}_{[i],1} \\
            \vdots  \\
            \Omega^{s_{ir_i}}\mathscr{L}_{[i],r_i-1}\\
            \mathscr{L}_{[i],r_i}
        \end{array}
    \right).
\]
Here, we recall the special series defined in (\ref{Eq:Deformation_Series})
\begin{align}
    \mathscr{L}_{[i],j}:=\sum_{\ell_1>\cdots>\ell_{j}\geq 0}
    (\Omega^{s_{ij}}Q_{ij})^{(\ell_j)}\cdots(\Omega^{s_{i1}}Q_{i1})^{(\ell_1)}\in\mathcal{E}.
\end{align}
We remark that this series satisfies the Frobenius difference equation:
\begin{align}\label{dL}
    \mathscr{L}_{[i],j}^{(-1)}=\mathscr{L}_{[i],j}+(\Omega^{s_{ij}}Q_{s_{ij}-1})^{(-1)}\mathscr{L}_{[i],j-1}.
\end{align}
Here, we set $\mathscr{L}_{[i], 0}=1$.
For later convenience, we set $\mathscr{L}_{[i]}:=\mathscr{L}_{[i],r_i}$. Importantly, the deformation series $\mathscr{L}_{[i],j}$ has the property that
\begin{equation}\label{Eq:Specialization_of_L}
    \mathscr{L}_{[i],j}(\theta^{q^N})=(\mathscr{L}_{[i],j}(\theta))^{q^N}
\end{equation}
for all $N\in\mathbb{Z}_{>0}$ (See \cite[Lem.~5.3.5]{C14},~\cite[Prop.~2.3.3]{CPY19}).

To simplify our notation, for $r_i\geq 2$, we express 
\[
\Phi_{[i]}=\left(
\begin{array}{c|c|c}
 (t-\theta)^w & 0 & 0\\ \hline
   D_{[i]} & \Phi_{[i]}'' & 0\\ \hline
    0 & \nu_{[i]} & 1
\end{array}
\right),
\Phi_{[i]}'=
\left(
    \begin{array}{c|c}
    (t-\theta)^w & 0\\ \hline
    D_{[i]} & \Phi_{[i]}''\\
    \end{array}
\right).
\]
where $D_{[i]}\in\Mat_{(r_i-1)\times 1}(\overline{k}[t])$, 
$\nu_{[i]}\in\Mat_{1\times (r_i-1)}(\overline{k}[t])$
and $\Phi_{[i]}''\in\Mat_{r_i-1}(\overline{k}[t])$.
We further set
\[
    \psi_{[i]}=
        \left(
        \begin{array}{clr}
            \Omega^{w}  \\ \hline
            \psi_{[i]}''\\ \hline
            \mathscr{L}_{[i]}
        \end{array}
    \right),~
    \psi_{[i]}'=
        \left(
        \begin{array}{clr}
            \Omega^{w}  \\ \hline
            \psi_{[i]}''
        \end{array}
    \right)
\]
where $\psi_{[i]}''\in\Mat_{(r_i-1)\times 1}(\mathcal{E})$.

For $r_i=1$, we express
\[
    \Phi_{[i]}=
        \begin{pmatrix}
            (t-\theta)^{w} & 0\\
            \nu_{[i]} & 1
        \end{pmatrix},~
    \Phi_{[i]}'=
        \begin{pmatrix}
            (t-\theta)^{w}
        \end{pmatrix}
\]
and
\[
    \psi_{[i]}=
        \left(
        \begin{array}{clr}
            \Omega^{w}  \\
            \mathscr{L}_{[i]}
        \end{array}
        \right),~
    \psi_{[i]}'=
        \begin{pmatrix}
            \Omega^w
        \end{pmatrix}.
\]

For each $H=(h_1,\dots,h_m)\in\Mat_{1\times m}(\mathcal{E})$, we denote by $\mathrm{ord}_\alpha(H):=\{\mathrm{ord}_\alpha(h_j)\}_{j=1}^m$ for $\alpha\in\mathbb{C}_\infty$. The following lemma is crucial in the proof of our main theorem.

\begin{lem}\label{Eq:Vanishing_Order_of_psi''}
    Suppose that $r_i\geq 2$, $\mathscr{L}_{[i],j}(\theta)\neq 0$ for all $1\leq i\leq m$ and $1\leq j\leq r_i-1$. We thus have
    $$\mathrm{ord}_{\theta^{q^N}}(\psi''_{[i]})=\{w-s_{i1},\dots,w-s_{i1}-\cdots-s_{ir_i-1}\}=g(\mathfrak{s}_i)$$
    for all $N\in\mathbb{Z}_{> 0}$.
\end{lem}

\begin{proof}
    Recall from the definition that
    \[
    \psi_{[i]}''=
        \left(
        \begin{array}{clr}
            \Omega^{s_{i2}+\cdots+s_{ir_i}}\mathscr{L}_{[i],1} \\
            \vdots  \\
            \Omega^{s_{ir_i}}\mathscr{L}_{[i],r_i-1}
        \end{array}
        \right).
    \]
    Since $\mathscr{L}_{[i],j}(\theta)\neq 0$ for all $1\leq i\leq m$ and $1\leq j\leq r_i-1$, we can use (\ref{Eq:Specialization_of_L}) to deduce that
    $$\mathrm{ord}_{\theta^{q^N}}(\Omega^{s_{i(j+1)}+\cdots+s_{ir_i}}\mathscr{L}_{[i],j})=\mathrm{ord}_{\theta^{q^N}}(\Omega^{s_{i(j+1)}+\cdots+s_{ir_i}})=w-s_{i1}-\cdots-s_{ij}.$$
    Then, the desired result immediately follows.
\end{proof}

Now, we consider the Frobenius difference equation associated with $\mathbb{F}_q[t]$-linear combinations $a_0\mathscr{L}_{[0]}+\cdots +a_m\mathscr{L}_{[m]}$ for some $a_0,\dots,a_m\in\mathbb{F}_q[t]$. The result is given by
\begin{equation}\label{Eq:Fiber_Phi}
    \Phi_\star:=
    \begin{pmatrix}
    (t-\theta)^{w} & 0 & 0 & \cdots & 0 \\
    D_{[1]} & \Phi_{[1]}'' & 0 & \cdots & 0\\
    \vdots &  & \ddots &  & \vdots \\
    D_{[m]} &  &  & \Phi_{[m]}'' & 0\\
    a_0\nu_{[0]} & \cdots & \cdots & a_m\nu_{[m]} & 1
    \end{pmatrix}
\end{equation}
and
\begin{equation}\label{Eq:Fiber_psi}
    \psi_\star:=
    \left(
        \begin{array}{clr}
            \Omega^{w}  \\
            \psi_{[1]}''\\
            \vdots\\
            \psi_{[m]}''\\
            a_0\mathscr{L}_{[0]}+\cdots+a_m\mathscr{L}_{[m]}
        \end{array}
    \right).
\end{equation}
One can check that $\psi_\star^{(-1)}=\Phi_\star\psi_\star$ by direct computation using \eqref{domega} and \eqref{dL}. To simplify the notation, we set
\[
    \Phi_\star=
    \begin{pmatrix}
    \Phi_\star' & 0\\
    \nu_\star & 1
    \end{pmatrix}.
\]
Let $M'_{[i]}$ (resp. $M'_\star$) be the Frobenius module defined by $\Phi'_{[i]}$ (resp. $\Phi'_\star$). Then, one can check directly that $M'_{[i]}$ defines an Anderson dual $t$-motive and $M'_\star$ is the fiber coproduct \cite[Sec.~2.4]{CM20} of $\{M'_{[i]}\}$ over $C^{\otimes w}$. Therefore, $M'_\star$ also defines an Anderson dual $t$-motive by \cite[Prop.~2.4.5]{CM20}.

\subsection{The $\mathrm{Ext}^1$-module and Anderson $t$-modules}\label{t-module}
In this section, we recall the isomorphism between certain Ext$^1$-modules and Anderson $t$-modules due to Anderson's idea. 

First, we review the definition of Anderson $t$-modules. Let $L$ be an $A$-field with $A\subset L\subset \mathbb{C}_\infty$, and let $\tau:=(x\mapsto x^{q}):L\to L$ be the Frobenius $q$-th power operator. Let $L[\tau]$ be the twisted polynomial ring in $\tau$ over $L$ subject to the relation $\tau\alpha=\alpha^q\tau$ for $\alpha\in L$. 
\begin{defn}[\cite{A86}]\label{Definition of t-module}
Let $L$ be an $A$-field with $A\subset L\subset \mathbb{C}_\infty$. For a fixed $d\in\mathbb{Z}_{>0}$, a $d$-dimensional Anderson $t$-module defined over $L$ is a pair $E=(\mathbb{G}_a^d,\rho)$ where $\mathbb{G}_a^d$ is the $d$-dimensional additive group scheme over $L$ and $\rho$ is an $\mathbb{F}_q$-linear ring homomorphism 
\begin{align*}
    \rho:\mathbb{F}_q[t]&\rightarrow \Mat_d(L[\tau])\\
                    a&\mapsto \rho_{a}
\end{align*}
such that when we write $\rho_t=\alpha_0+\sum_i\alpha_i\tau^i$ with $\alpha_i\in\Mat_d(L)$,  $\alpha_0-\theta I_d$ is a nilpotent matrix. A morphism of Anderson $t$-modules defined over $L$ is a morphism of additive group schemes over $L$ commuting with $\mathbb{F}_q[t]$-action. That is, for Anderson $t$-modules $E=(\mathbb{G}_a^d, \rho)$ and $F=(\mathbb{G}_a^m, \mu)$ over $L$, the morphism $f$ satisfies $f\rho_a=\mu_a f$ $(a\in\mathbb{F}_q[t])$.
\end{defn}
Later on, supposing that $L\subset F$ is a field extension, we consider the $F$-valued points
of the Anderson $t$-module $E$ defined over $L$ and denote it by $E(F)$: that is, 
a pair $(\mathbb{G}_a^d(F), \rho)$ of the $F$-valued points of $\mathbb{G}_a^d$ and the $\mathbb{F}_q$-linear ring homomorphism $\rho$ such that $\rho(\mathbb{F}_q[t])\subset\Mat_d(F[\tau])$. 

We fix an Anderson dual $t$-motive $M'$ of rank $d$ over $\overline{k}[t]$, which is defined by the matrix $\Phi'\in\Mat_{d}(\overline{k}[t])$. 
Then, we define $\Ext_{\mathscr{F}}^1({\bf 1}, M')$ to be the left $\mathbb{F}_q[t]$-module of equivalence classes $[ M]$ of Frobenius modules $M$, which fits into the following short exact sequence of Frobenius modules:
\[
    0\rightarrow M'\rightarrow M\rightarrow {\bf 1}\rightarrow 0.
\]

The left $\mathbb{F}_q[t]$-module structure of $\Ext_{\mathscr{F}}^1({\bf 1}, M')$ comes from the Baer sum and pushout of morphisms of $M'$. 
Specifically, the structure is described as follows. We assume that $M_1, M_2 \in \mathscr{F}$ are defined by 
\[
\begin{pmatrix}
    \Phi' & 0\\
    {\bf v}_1 & 1\\
    \end{pmatrix}
,
\begin{pmatrix}
    \Phi' & 0\\
    {\bf v}_2 & 1\\
    \end{pmatrix}
    \in\Mat_{d+1}(\overline{k}[t])
\]    
respectively, and both fit into the above short exact sequence.    
Then, the Baer sum $M_1+_{B}M_2$ of $M_1$ and $M_2$ is a Frobenius module defined by the matrix
\[
\begin{pmatrix}
    \Phi' & 0\\
    {\bf v}_1+{\bf v}_2 & 1\\
    \end{pmatrix}.
\]

Moreover, by taking the scalar product with $a\in \mathbb{F}_q[t]$, we obtain an endomorphism $a:M'\rightarrow M'$. Then, the pushout $a\ast M_1$ of the endomorphism is a Frobenius module defined by
\[
\begin{pmatrix}
    \Phi' & 0\\
    a{\bf v}_1 & 1\\
    \end{pmatrix}.
\]
Then, $\Ext^1_{\mathscr{F}}({\bf 1}, M')$ forms an $\mathbb{F}_q[t]$-module with 
\[
     [M_1]+[M_2]:=[M_1+_B M_2],\quad a[M_1]:=[a\ast M_1] \quad (a\in\mathbb{F}_q[t])
\]
for representatives $[M_1], [M_2]\in\Ext^1_{\mathscr{F}}({\bf 1}, M')$.
Anderson proved the following result involving $\Ext^1_{\mathscr{F}}({\bf 1}, M')$.   
\begin{thm}[Anderson,~{\cite[Thm.~5.2.1]{CPY19}}]\label{Thm:Anderson_Ext}
Let $M'$ be an Anderson dual $t$-motive. Then, we have the following $\mathbb{F}_q[t]$-module isomorphism:
\begin{equation}\label{Eq:F_q[t]-module_Isomorphism}
    \Ext_{\mathscr{F}}^1(\mathbf{1},M')\cong M'/(\sigma-1)M'\cong E'(\overline{k})
\end{equation}
where $E'$ is the Anderson $t$-module associated with $M'$ defined over $\overline{k}$ in the sense that the $\overline{k}$-valued point of $E'$ is isomorphic to $M'/(\sigma-1)M'$ as $\mathbb{F}_q$-vector spaces and the $\mathbb{F}_q[t]$-module structure on $E'$ via $\rho$ is induced by the $\mathbb{F}_q[t]$-action on $M'/(\sigma-1)M'$.  
\end{thm}

For more details about the construction of these isomorphisms, see \cite[Sec.~5.2]{CPY19}. We also refer readers to \cite{PR03}, \cite{HP04}, \cite{Ta10} and \cite{HJ16} for a related discussion.

For the $n$-th tensor power of the Carlitz motive, the isomorphisms \eqref{Eq:F_q[t]-module_Isomorphism} are described in the following example.

\begin{eg}
    Let $n\in\mathbb{Z}_{>0}$. The {\it$n$-th tensor power of the Carlitz module} is the associated Anderson $t$-module of $C^{\otimes n}$, which is given by $\mathbf{C}^{\otimes n}:=(\mathbb{G}_a^n,[\cdot]_n)$, where $[\cdot]_n$ is uniquely determined by
    \[
        [t]_n:=
        \begin{pmatrix}
        \theta & 1 & \cdots & 0 \\
         & \theta & \ddots & \vdots\\
         & & \ddots & 1 \\
        \tau & & & \theta
        \end{pmatrix}\in\Mat_n(\mathbb{C}_\infty[\tau]).
    \]
    Recall that $C^{\otimes n}$ has a $\overline{k}[\sigma]$-basis $\{(t-\theta)^{n-1},\dots,(t-\theta),1\}$; thus, every $f\in C^{\otimes{n}}/(\sigma-1)C^{\otimes n}$ has a unique representative with $t$-degree less than or equal to $n-1$ of the form
    $$f_1(t-\theta)^{n-1}+f_2(t-\theta)^{n-2}+\cdots+f_n,~f_i\in\overline{k}.$$
    Then, the isomorphisms \eqref{Eq:F_q[t]-module_Isomorphism} can be explicitly described as follows:
    \begin{align*}
        \Ext_{\mathscr{F}}^1(\mathbf{1},C^{\otimes n})&\cong C^{\otimes n}/(\sigma-1)C^{\otimes n}\cong \mathbf{C}^{\otimes n}(\overline{k})\\
        [M_f]&\mapsto f+(\sigma-1)C^{\otimes n}\mapsto (f_1,\dots,f_n)^{\mathrm{tr}}
    \end{align*}
    where $M_f$ is the Frobenius module defined by the matrix
    \[
        \Phi_f:=
        \begin{pmatrix}
            (t-\theta)^n & 0 \\
            f^{(-1)}(t-\theta)^n & 1
        \end{pmatrix}.
    \]
\end{eg}

Next, we present two lemmas that enable us to consider linear independence of certain special values via the torsion elements of Anderson $t$-modules and $\Ext^1$-modules. The first lemma concerns the $\overline{k}(t)$-linear independence of elements in $\mathcal{E}$.

\begin{lem}\label{Lem:Vanishing_of_Entire_Functions}
    Let $f_1,\dots,f_m\in\mathcal{E}$ be non-zero elements and $c_1, \dots,c_m\in\overline{k}(t)$ such that $c_1f_1+\cdots+c_mf_m=0$. Suppose that there are infinitely many $\alpha\in\mathbb{C}_\infty$ such that $\mathrm{ord}_\alpha(f_i)$ is different for each $i$. Then, we have $c_i=0$ for all $1\leq i\leq m$.
\end{lem}

\begin{proof}
    We prove this lemma by induction. In the case $m=1$, on the basis of the assumption that $f_1\neq 0$ and $f_1\in\mathcal{E}$, it follows that there are infinitely many $\alpha\in\mathbb{C}_\infty$ that satisfy $f_1|_{t=\alpha}\neq 0$ or otherwise $f_1$ must vanish identically by the entireness. Indeed, non-zero elements in $\mathcal{E}$ only admit finitely many zeros in any bounded disc centred on $0$ according to \cite[Prop.2.11]{Go96}. Therefore, there are infinitely many $\alpha\in\mathbb{C}_\infty$ such that $c_1|_{t=\alpha}=0$ since $c_1f_1=0$. Then, because $c_1\in\overline{k}(t)$ and every non-zero element in $\overline{k}(t)$ has only finitely many zeros, we conclude that $c_1$ vanishes identically and we complete the case of $m=1$. In the case $m=N$, based on the assumptions, we can choose $1\leq j\leq N$ such that there are infinitely many $\alpha\in\mathbb{C}_\infty$ with $\mathrm{ord}_\alpha(f_j)<\mathrm{ord}_\alpha(f_i)$ for all $1\leq i\neq j\leq N$. Without loss of generality, we may assume that $j=N$. Since $c_1,\dots,c_N$ are rational functions, they only admit finitely many poles. Thus, there are infinitely many $\alpha\in\mathbb{C}_\infty$ such that
    \[
        (c_1f_1+\cdots+c_Nf_N)(t-\alpha)^{-\mathrm{ord}_\alpha(f_N)}
    \]
    is defined at $t=\alpha$. By evaluating the above quantity at $t=\alpha$, we obtain $c_N|_{t=\alpha}=0$ for infinitely many $\alpha\in\mathbb{C}_\infty$ since
    \[
    f_i(t-\alpha)^{-\mathrm{ord}_\alpha(f_N)}\mid_{t=\alpha}=0
    \]
    for all $1\leq i<N$ and
    \[
    f_N(t-\alpha)^{-\mathrm{ord}_\alpha(f_N)}\mid_{t=\alpha}\neq 0.
    \]
    Hence,  $c_N$ vanishes identically, and the desired result  follows immediately from the induction hypothesis.
\end{proof}

In what follows, we establish a criterion for elements being a torsion element in certain $\Ext^1$-modules.

\begin{lem}[cf.~{\cite[Thm.~2.5.2]{CPY19}}]\label{Lem:Torsion_of_Ext}
    Let $S=\{\mathfrak{s}_0, \dots, \mathfrak{s}_m\}\subset I(w)$ be $g$-independent with $\mathfrak{s}_0=(w)$, and let $\mathfrak{Q}_i\in\mathbb{D}_{\mathfrak{s}_i}$ such that $\mathscr{L}_{[i],j}(\theta)\neq 0$ for each $0\leq i\leq m$ and $1\leq j\leq r_i$. Let $M_\star$ be the Frobenius module defined by $\Phi_\star$. If $a_0\mathscr{L}_{[0]}+\cdots + a_m\mathscr{L}_{[m]}$ vanishes at $t=\theta$, then $[M_\star]$ is an $\mathbb{F}_q[t]$-torsion element in $\Ext_\mathscr{F}^1(\mathbf{1},M'_\star)$.
\end{lem}

\begin{proof}
    It suffices to show that there exists a non-zero $b\in\mathbb{F}_q[t]$ such that $b[M_\star]$ represents the trivial class in $\Ext_\mathscr{F}^1(\mathbf{1},M'_\star)$. 
    Consider the associated Frobenius difference equation $\psi_\star^{(-1)}=\Phi_\star\psi_\star$ defined in (\ref{Eq:Fiber_Phi}) and (\ref{Eq:Fiber_psi}). 
    Then, according to \cite[Thm. ~3.1.1]{ABP04}, there exists $\mathbf{f}=(f_0,\vec{f}_1,\dots,\vec{f}_m,f_{m+1})$ with $f_0,f_{m+1}\in \overline{k}[t]$ and $\vec{f}_i\in\Mat_{1\times (r_i-1)}(\overline{k}[t])$ for each $1\leq i\leq m$ such that $\mathbf{f}\psi_\star=0$ and $\mathbf{f}(\theta)=(0,\dots,0,1)$. We set $\mathbf{f}':=(f'_0,\vec{f'}_1,\dots,\vec{f'}_m,1)$, where 
    $$f'_0=f_0/f_{m+1}\in\overline{k}(t) \mbox{ and } \vec{f'}_i=\frac{1}{f_{m+1}}\vec{f}_i\in\Mat_{1\times (r_i-1)}(\overline{k}(t))$$
    for each $1\leq i\leq m-1$. Note that $\mathbf{f}'\psi_\star=0$; thus, $(\mathbf{f}'-\mathbf{f}'^{(-1)}\Phi_\star)\psi_\star=0$. Therefore, by setting
    $$\mathbf{f}'-\mathbf{f}'^{(-1)}\Phi_\star=(R_0,\vec{R}_1,\dots,\vec{R}_m,0),$$
    we obtain
    $$R_0\Omega^w+\vec{R}_1\cdot\psi''_{[1]}+\cdots+\vec{R}_m\cdot\psi''_{[m]}=0.$$
    Lemma~\ref{Eq:Vanishing_Order_of_psi''} yields
    $$\mathrm{ord}_{\theta^{q^N}}(\psi''_{[i]})\cap\mathrm{ord}_{\theta^{q^N}}(\psi''_{[j]})=g(\mathfrak{s}_i)\cap g(\mathfrak{s}_j)=\emptyset$$
    for all $N\in\mathbb{Z}_{\geq 0}$ $1\leq i\neq j\leq m$. One can now check directly that $$\mathrm{ord}_{\theta^{q^N}}\Bigl((\Omega^w,\psi''_{[1]},\dots,\psi''_{[m]})^{\mathrm{tr}}\Bigr)=\{w\}\cup g(\mathfrak{s}_1)\cup\cdots\cup g(\mathfrak{s}_m)$$ is a disjoint union for all $N\in\mathbb{Z}_{\geq 0}$. Then, Lemma~\ref{Lem:Vanishing_of_Entire_Functions} shows that $R_0=0$ and $\vec{R}_i=\vec{0}$ for all $1\leq i\leq m$. If we consider
    $$\gamma:=
    \begin{pmatrix}
        1 &  &  &  &  \\
         & \mathbb{I}_{r_1-1} &  &  & \\
         &  & \ddots &  &  \\
         &  &  & \mathbb{I}_{r_m-1} & \\
        f'_0 & \vec{f}'_1 & \cdots & \vec{f}'_m & 1
    \end{pmatrix},$$
    where $\mathbb{I}_n$ denote by $n\times n$ identity matrix, then we have
    $$\gamma^{(-1)}\Phi_\star=
    \begin{pmatrix}
        \Phi'_\star &  \\
         & 1 
    \end{pmatrix}\gamma.$$
    Now, according to \cite[Prop.~2.2.1]{CPY19}, there exists a non-zero $b\in\mathbb{F}_q[t]$ such that $bf'_0\in\overline{k}[t]$ and $b\vec{f}'_i\in\Mat_{1\times (r_i-1)}(\overline{k}[t])$ for all $1\leq i\leq m$. If we consider $\nu=(bf'_0,b\vec{f}'_1,\dots,b\vec{f}'_m)$ and 
    $$\delta:=\begin{pmatrix}
        1 &  &  &  &  \\
         & \mathbb{I}_{r_1-1} &  &  & \\
         &  & \ddots &  &  \\
         &  &  & \mathbb{I}_{r_m-1} & \\
        bf'_0 & b\vec{f}'_1 & \cdots & b\vec{f}'_m & 1
    \end{pmatrix},$$
    then we have
    $$\delta^{(-1)}
    \begin{pmatrix}
        \Phi_\star &  \\
        \nu & 1 
    \end{pmatrix}=
    \begin{pmatrix}
        \Phi'_\star &  \\
         & 1 
    \end{pmatrix}\delta.$$
    Consequently, by changing the $\overline{k}[t]$-basis of $b\ast M_\star$ with $\delta$, we conclude that $b [M_\star]$ represents the trivial class in $\Ext^1_{\mathscr{F}}(\mathbf{1},M'_\star)$, and the desired result follows immediately.
\end{proof}

\subsection{Anderson-Thakur polynomials and special values}\label{AT-polynomial}
    Before we move to the main theorem, let us first briefly review Anderson-Thakur polynomials~\cite{AT90}. These polynomials are needed to describe the applications of Theorem \ref{Thm:Main_Thm} to $\infty$-adic MZVs, $\infty$-adic AMZVs and CMPLs at algebraic points. Set $F_0:=1$ and $F_i:=\underset{j=1}{\overset{i}{\prod}}(t^{q^i}-\theta^{q^j})$ and define the Anderson-Thakur polynomials $H_n\in A[t]$ by the following generating function: 
    $$\left(1-\underset{i=0}{\overset{\infty}{\sum}}\frac{F_i}{D_i\mid_{\theta=t}}x^{q^i}\right)^{-1}=\underset{n=0}{\overset{\infty}{\sum}}\frac{H_n}{\Gamma_{n+1}\mid_{\theta=t}}x^n.$$
In \cite[Thm. 3.3]{Ch17}, $H_n$ is explicitly given for some specific $n\in\mathbb{Z}$, as follows:
\begin{eg}
\begin{align*}
&H_{0}(t)=1,\ \ H_{q^2-q-1}(t)=\Gamma_{q^2-q}|_{\theta=t},\ \ H_{q^3-1}(t)=\Gamma_{q^3}|_{\theta=t},\  H_{q^2-q}(t)=\Gamma_{q^2-q+1}|_{\theta=t}\frac{(t-\theta^q)^{q-1}}{L_1^{q-1}|_{\theta=t}}. 
\end{align*}
\end{eg}
    
    Let $n\in\mathbb{Z}_{>0}$ and $i\in\mathbb{Z}_{\geq 0}$. We set the power sum to be 
    $$S_i(n):=\sum_{a\in A_{i+}}\frac{1}{a^n}\in k .$$
    Then, an important identity established in \cite{AT90} is the following:
    \begin{equation}\label{Eq:Identity_of_AT_polynomials}
        (\Omega^nH_{n-1})^{(i)}\mid_{t=\theta}=\frac{\Gamma_nS_i(n)}{\Tilde{\pi}^n}.
    \end{equation}
    Based on (\ref{Eq:Identity_of_AT_polynomials}), it has been shown that many special values appear in the specialization of the deformation series given in (\ref{Eq:Deformation_Series}). We collect some results from \cite{AT90}, \cite{P08}, \cite{AT09}, \cite{C14}, \cite{C16}, \cite{CPY19}, and \cite{H20} to state the following proposition.
    
    \begin{prop}\label{Prop:Specialization}
        Let $\mathfrak{s}=(s_1,\dots,s_r)\in\mathbb{Z}_{>0}^r$ and $\mathfrak{Q}=(Q_1,\dots,Q_r)\in\overline{k}[t]^{r}$. We set 
        \[
            \mathscr{L}:=\mathscr{L}_{\mathfrak{s},\mathfrak{Q}}:=
            \sum_{\ell_1>\cdots>\ell_{r}\geq 0}
            (\Omega^{s_{r}}Q_{r})^{(\ell_r)}\cdots(\Omega^{s_{1}}Q_{1})^{(\ell_1)}\in\overline{k}\llbracket t\rrbracket.
        \]
        \begin{enumerate}
            \item If $\mathfrak{Q}=(u_1,\dots,u_r)\in\mathbb{D}_\mathfrak{s}\cap (\overline{k})^r$, then
            \[
                \mathscr{L}\mid_{t=\theta}=\frac{\Li_\mathfrak{s}(u_1,\dots,u_r)}{\Tilde{\pi}^{s_1+\cdots+s_r}}.
            \]
            \item If $\mathfrak{Q}=(H_{s_1-1},\dots,H_{s_r-1})$, then
            \[
                \mathscr{L}\mid_{t=\theta}=\frac{\Gamma_{s_1}\cdots\Gamma_{s_r}\zeta_A(\mathfrak{s})}{\Tilde{\pi}^{s_1+\cdots+s_r}}.
            \]
            \item If $\mathfrak{Q}=(\gamma_1H_{s_1-1},\dots,\gamma_rH_{s_r-1})$, where $\gamma_i$ is a fixed $(q-1)$-th root of $\epsilon_i\in\mathbb{F}_q^\times$ for $\boldsymbol{\epsilon}=(\epsilon_1,\dots,\epsilon_r)\in(\mathbb{F}_q^\times)^r$, then
            \[
                \mathscr{L}\mid_{t=\theta}=\frac{\gamma_1\cdots\gamma_r\Gamma_{s_1}\cdots\Gamma_{s_r}\zeta_A(\mathfrak{s};\boldsymbol{\epsilon})}{\Tilde{\pi}^{s_1+\cdots+s_r}}.
            \]
        \end{enumerate}
    \end{prop}
    
    \begin{rem}
        For our purpose, the identity of $\infty$-adic AMZVs given in Proposition~\ref{Prop:Specialization} is slightly different than the original version given in \cite[Thm.~3.4]{H20}. To derive Proposition~\ref{Prop:Specialization}~(3), we combine  (\ref{Eq:Identity_of_AT_polynomials}) with the fact that for each $\epsilon\in\mathbb{F}_q^\times$ and fixed $(q-1)$-th root $\gamma$ of $\epsilon$, we have $\gamma^{(i)}=\gamma\cdot\epsilon^i$ for all $i\in\mathbb{Z}_{\geq 0}$. Then, an argument similar to that of \cite[Thm.~3.4]{H20} gives the desired identity.
    \end{rem}

\section{The main result and applications}\label{Main and app}
In this section, we first state our main theorem and then present some applications.

\subsection{Main Theorem}\label{Main Theorem}
In what follows, we describe and prove our main result, a linear independence criterion for special values 
that appear in the specialization at $t=\theta$ of the deformation series given in (\ref{Eq:Deformation_Series}).

\begin{thm}\label{Thm:Main_Thm}
    For $w\in\mathbb{Z}_{>0}$, let $S=\{\mathfrak{s}_0,\dots, \mathfrak{s}_m\}\subset I(w)$ be $g$-independent with $\mathfrak{s}_0=(w)$ and let $\mathfrak{Q}_i\in\mathbb{D}_{\mathfrak{s}_i}$ such that $\mathscr{L}_{[i],j}(\theta)\neq 0$ for each $0\leq i\leq m$ and $1\leq j\leq r_i$.
    If we set $\mathscr{L}_{[i]}:=\mathscr{L}_{[i],\dep(\mathfrak{s}_i)}$, then the following set 
    \[
        \{ \mathscr{L}_{[0]}(\theta), \ldots,  \mathscr{L}_{[m]}(\theta)\}
    \]
    is $k$-linearly independent.
\end{thm}

\begin{proof}
    By contrast, suppose that $\{ \mathscr{L}_{[0]}(\theta), \ldots,  \mathscr{L}_{[m]}(\theta)\}$ is a $k$-linearly dependent set. Then, there exists $a_0,\dots,a_m\in\mathbb{F}_q[t]$ not all zero such that
    \begin{equation}\label{Eq:thm3.1}
        a_0(\theta)\mathscr{L}_{[0]}(\theta)+\cdots+a_m(\theta)\mathscr{L}_{[m]}(\theta)=0.
    \end{equation}
    
    Now, we consider the Frobenius difference equation associated with the $\mathbb{F}_q[t]$-linear combinations $a_0\mathscr{L}_{[0]}+\cdots+a_m\mathscr{L}_{[m]}$ given in (\ref{Eq:Fiber_Phi}) and (\ref{Eq:Fiber_psi}). Let $M_\star$ be the Frobenius module defined by $\Phi_\star$. Then, Lemma~\ref{Lem:Torsion_of_Ext} shows that $[M_\star]$ represents a torsion element in $\Ext^1_{\mathscr{F}}(\mathbf{1},M'_\star)$. Thus, there is a non-zero element $b\in\mathbb{F}_q[t]$ such that $b[M_\star]$ represents the trivial class in $\Ext^1_{\mathscr{F}}(\mathbf{1},M'_\star)$. Let $\{x_0,x_{11},\dots,x_{1r_1},\dots,x_{m1},\dots,x_{mr_m}\}$ be a $\overline{k}[t]$-basis of $M'_\star$ such that the $\sigma$-action is represented by $\Phi'_\star$. Then, we have the following $\mathbb{F}_q[t]$-module isomorphism (see \cite[Thm.~5.2.1]{CPY19}):
    \begin{align*}
        \Ext_\mathscr{F}^1(\mathbf{1},M'_\star)&\cong M'_\star/(\sigma-1)M'_\star\\
        [b\ast M_\star]&\mapsto ba_0Q_{01}^{(-1)}(t-\theta)^wx_0+\sum_{i=1}^mba_iQ_{ir_i}^{(-1)}(t-\theta)^{s_{ir_i}}x_{ir_i}+(\sigma-1)M'_\star.
    \end{align*}
    Now, we consider the natural projection map
    \begin{align*}
        \pi:M'_\star&\twoheadrightarrow\oplus_{i=1}^mC^{\otimes s_{ir_i}}\\
        g_0x_0+\sum_{i=1}^m\sum_{j=1}^{r_i}g_{ij}x_{ij}&\mapsto(g_{1r_1},\dots,g_{mr_m}).
    \end{align*}
    Since $\pi\circ(\sigma-1)=(\sigma-1)\circ\pi$ and $\pi$ is surjective, we deduce that
    \begin{equation}\label{Eq:Delta_Map}
        \Delta:E'_\star(\overline{k})\cong M'_\star/(\sigma-1)M'_\star\twoheadrightarrow\oplus_{i=1}^m\left(C^{\otimes s_{ir_i}}/(\sigma-1)C^{\otimes s_{ir_i}}\right)\cong\oplus_{i=1}^m\mathbf{C}^{\otimes s_{ir_i}}(\overline{k}).
    \end{equation}
    Let $\mathbf{v}_{[i]}$ be the point in $\mathbf{C}^{\otimes s_{ir_i}}(\overline{k})$ corresponding to  $Q_{ir_i}^{(-1)}(t-\theta)^{s_{ir_i}}$ in $C^{\otimes s_{ir_i}}/(\sigma-1)C^{\otimes s_{ir_i}}$ for each $1\leq i\leq m$. Note that $\mathbf{v}_{[i]}\neq 0$ for each $1\leq i\leq m$ by (\ref{Eq:Condition_of_Q_2}). 
    Then, the trivial class $b[M_\star]$ is mapped to $([ba_1]_{s_{1r_1}}\mathbf{v}_{[1]},\dots,[ba_m]_{s_{mr_m}}\mathbf{v}_{[m]})$ in $\oplus_{i=1}^m\mathbf{C}^{\otimes s_{ir_i}}(\overline{k})$ by $\Delta$ in (\ref{Eq:Delta_Map}). If there is an $a_i$ that is non-zero, then $\mathbf{v}_{[i]}$ is a non-zero $\mathbb{F}_q[t]$-torsion element in $\mathbf{C}^{\otimes s_{ir_i}}(\overline{k})$ because $b\neq 0$, $a_i\neq 0$ and $\mathbf{v}_{[i]}\neq 0$, while $[ba_i]_{s_{ir_i}}\mathbf{v}_{[i]}=0$. This situation leads to a contradiction since $\mathbf{v}_{[i]}$ is not a torsion element according to (\ref{Eq:Condition_of_Q_2}).
    Consequently, we must have $a_i=0$ for each $1\leq i\leq m$. 
    Thus, by \eqref{Eq:thm3.1}, $a_0$ also vanishes, and we obtain the desired result.
\end{proof}

\begin{rem}\label{Rem:MZ_property}
    Notably, the $k$-linear independence of $\{ \mathscr{L}_{[0]}(\theta), \ldots,  \mathscr{L}_{[m]}(\theta)\}$ is equivalent to the $k$-linear independence of $\{ \tilde{\pi}^w\mathscr{L}_{[0]}(\theta), \ldots, \tilde{\pi}^w\mathscr{L}_{[m]}(\theta)\}$. Also, we note that each element in $\{ \tilde{\pi}^w\mathscr{L}_{[0]}(\theta), \ldots, \tilde{\pi}^w\mathscr{L}_{[m]}(\theta)\}$ has the MZ (multizeta) property with weight $w$ in the sense of \cite[Def.~3.4.1]{C14} (cf.~\cite[Sec.~4.2]{H20}). Hence, on the basis of \cite[Prop.~4.3.1]{C14}, we conclude that the $k$-linear independence of $\{ \mathscr{L}_{[0]}(\theta), \ldots,  \mathscr{L}_{[m]}(\theta)\}$ is equivalent to the $\overline{k}$-linear independence of $\{ \mathscr{L}_{[0]}(\theta), \ldots,  \mathscr{L}_{[m]}(\theta)\}$.
\end{rem}

\begin{rem}
    Let $r\in\mathbb{Z}_{>0}$. Let $\mathfrak{s}=(s_1,\dots,s_r)\in\mathbb{Z}_{>0}^r$ and $\mathfrak{Q}=(Q_1,\dots,Q_r)\in\overline{k}[t]^r$. The difficulty in applying Theorem~\ref{Thm:Main_Thm} is checking whether $\mathfrak{Q}$ satisfies (\ref{Eq:Condition_of_Q_2}). Note that checking condition (\ref{Eq:Condition_of_Q_2}) is equivalent to showing that the corresponding point of $Q_r$ in $C^{\otimes s_r}/(\sigma-1)C^{\otimes s_r}\cong \mathbf{C}^{\otimes s_r}(\overline{k})$ is not an $\mathbb{F}_q[t]$-torsion point. In general, it is not easy to determine whether a point is an $\mathbb{F}_q[t]$-torsion in the tensor powers of a Carlitz module. However, in the case of $\mathfrak{Q}=(Q_1,\dots,Q_r)\in k[t]^r$, we may use \cite[Prop.~1.11.2]{AT90} to study the $k$-rational $\mathbb{F}_q[t]$-torsion points in the tensor powers of a Carlitz module. A detailed discussion of this topic is presented in the next subsection.
\end{rem}

\subsection{Some applications}
    As an application, we describe how to associate a partition of $\{1,\dots,w-1\}$ to a $k$-linearly independent set of  weight $w$ special values, including MZVs. Consider the inverse map of $g$
\begin{align*}
    g^{-1}:g(I(w))\subset J(w)&\rightarrow I(w)\\
    \{x_1>\cdots>x_{r-1}\}&\mapsto (w-x_1, x_1-x_2, \ldots, x_{r-2}-x_{r-1}, x_{r-1} ). 
\end{align*}

A partition $P$ of $\{1,\dots,w-1\}$ is a subset of $J(w)$ such that for all $\mathcal{P}\in P$ and $\mathcal{P}'\in P$, it satisfies that $\mathcal{P}\cap\mathcal{P}'=\emptyset$ and $\bigcup_{\mathcal{P}\in P}\mathcal{P}=\{1,\dots,w-1\}$. A partition $P$ of $\{1,\dots,w-1\}$ is called $q$-admissible if for each $\mathcal{P}\in P$ the minimal element of $\mathcal{P}$ is not divisible by $q-1$. For example, we can give the following $k$-linearly independent set of $\infty$-adic MZVs and CMPLs at rational points whose indices come from the $q$-admissible partition.

\begin{prop}\label{Prop:Non-Torsion}
    For $w\in\mathbb{Z}_{>0}$, let $S=\{\mathfrak{s}_0,\dots, \mathfrak{s}_m\}\subset I(w)$ be $g$-independent with $\mathfrak{s}_0=(w)$ and $\mathfrak{s}_i=(s_{i1},\dots,s_{ir_i})$. Suppose that $r_i:=\dep(\mathfrak{s}_i)$ and $s_{ir_i}$ is not divisible by $q-1$; then, $\mathfrak{Q}_i=(Q_{i1},\dots,Q_{ir_i})\in k[t]^r$ satisfying (\ref{Eq:Condition_of_Q_1}) automatically satisfies (\ref{Eq:Condition_of_Q_2}). In particular, let $P=\{\mathcal{P}_1,\dots,\mathcal{P}_m\}$ be a $q$-admissible partition of $\{1,\dots,w-1\}$. We set $\mathfrak{s}_0:=(w)$ and $\mathfrak{s}_i:=g^{-1}(\mathcal{P}_i)$ for $1\leq i\leq m$. Then, the following collection of $\infty$-adic MZVs of weight $w$
    $$\{\zeta_A(w),\zeta_A(\mathfrak{s}_1),\dots,\zeta_A(\mathfrak{s}_m)\}$$
    is a $k$-linear independent set.
\end{prop}

\begin{proof}
    Since we know that the $k$-rational torsion elements $\mathbf{C}^{\otimes{s_{ir_i}}}(k)_{\mathrm{tor}}=\{0\}$ unless $q-1$ divides $s_{ir_i}$ according to \cite[Prop.~1.11.2]{AT90}, the assumption $s_{ir_i}$ is not divisible by $q-1$ implies that $\mathfrak{Q}_i=(Q_{i1},\dots,Q_{ir_i})\in k[t]^r$ such that $\mathfrak{Q}_i$ satisfying (\ref{Eq:Condition_of_Q_1}) automatically satisfies (\ref{Eq:Condition_of_Q_2}). Thus, we complete the first assertion.
    
    For the second assertion, we first note that $P$ being $q$-admissible implies that $s_{ir_i}$ is not divisible by $q-1$. Additionally, it is clear that $S=\{\mathfrak{s}_0=(w),\mathfrak{s}_1,\dots,\mathfrak{s}_m\}\subset I(w)$ forms a $g$-independent set. Thus, the desired $k$-linearly independent result of $\infty$-adic MZVs follows directly from Proposition~\ref{Prop:Specialization}, Theorem~\ref{Thm:Main_Thm}, and the first assertion.
\end{proof}

\begin{rem}
    For $w\in\mathbb{Z}_{>0}$, let $S$ be the collection of all indices of depth less than or equal to two and of weight $w$ such that the last entry of each index in $S$ is not divisible by $q-1$. Then, the same argument as that in the proof of Proposition~\ref{Prop:Non-Torsion} recovers \cite[Thm~3.1.1]{C16}.
\end{rem}

Next, we present another application of Theorem \ref{Thm:Main_Thm}. 
Let $r,~w\in\mathbb{Z}_{>0}$ and $\mathcal{Z}_w^r$ be the $k$-vector space spanned by $\infty$-adic MZVs of weight $w$ and depth $r$. Also, recall that $\mathcal{Z}_w^{1,r}=\mathcal{Z}_w^{1}+\mathcal{Z}_w^r$. Then, we have the following: 

\begin{cor}\label{Cor:Lower_Bound}
    Let $w,r\in\mathbb{Z}_{>0}$ be given. Then, we have
    $$\dim_k\mathcal{Z}_{w}^{1,r}\geq 1+ \lfloor\frac{w-1-\lfloor\frac{w-1}{q-1}\rfloor}{r-1}\rfloor$$
    where $\lfloor -\rfloor$ is the floor function. In particular,
    $$\dim_k\mathcal{Z}_{w}^r\geq \lfloor\frac{w-1-\lfloor\frac{w-1}{q-1}\rfloor}{r-1}\rfloor.$$
\end{cor}

\begin{proof}
    Consider the set 
    $$\mathcal{S}:=\{n\in\mathbb{Z}_{>0}\mid 1\leq n\leq w-1,~n~\mbox{is not divisible by}~q-1\}.$$ 
    Clearly, the cardinality of $\mathcal{S}$ is $$|\mathcal{S}|:=w-1-\lfloor\frac{w-1}{q-1}\rfloor.$$ 
    Let $\ell:=\lfloor|\mathcal{S}|/(r-1)\rfloor$. Then, it is clear that we can find $\ell$ many subsets $\mathcal{S}_1,\dots,\mathcal{S}_\ell$ of $\mathcal{S}$ with cardinality $r-1$ and $\mathcal{S}_i\cap\mathcal{S}_j=\emptyset$ for all $i\neq j$. Now, we set $\mathfrak{s}_i:=g^{-1}(\mathcal{S}_i)$ for all $1\leq i\leq\ell$. Then, $\dep(\mathfrak{s}_i)=r$ and $\{\mathfrak{s}_0:=(w),\mathfrak{s}_1,\dots,\mathfrak{s}_\ell\}$ forms a $g$-independent set. Thus, Theorem~\ref{Thm:Main_Thm} together with the second part of Proposition~\ref{Prop:Specialization} shows that $\{\zeta_A(w),\zeta_A(\mathfrak{s}_1),\dots,\zeta_A(\mathfrak{s}_\ell)\}$ is a $k$-linearly independent set. The desired lower bound of $\dim_k\mathcal{Z}_{w}^{1,r}$ follows immediately.
\end{proof}

We obtain the following linearly independent set between the same-weight $\infty$-adic MZVs and the CMPLs at rational points using Proposition~\ref{Prop:Non-Torsion}.

\begin{eg}
    Let $q=5$, $w=6$, $\mathcal{P}_1=\{1,3,5\}$ and $\mathcal{P}_2=\{2,4\}$. Then, $P=\{\mathcal{P}_1,\mathcal{P}_2\}$ is a $q$-admissible partition of $\{1,\dots,5\}$. Let $\mathfrak{s}_1=g^{-1}(\mathcal{P}_1)=(1,2,2,1)$ and $\mathfrak{s}_2=g^{-1}(\mathcal{P}_2)=(2,2,2)$. Let $u\in k$ such that $\Li_6(z)$ converges at $z=u$ and $\Li_6(u)\neq 0$. Then, 
    \[
        \{\Li_6(u),\zeta_A(1,2,2,1),\zeta_A(2,2,2)\}
    \]
    is a $k$-linearly independent set.
\end{eg}

We end our exposition by noting some further studies on the algebraic independence of $\infty$-adic MZVs by Chang and Yu \cite{CY07} and Mishiba \cite{Mi15a, Mi17}.
\begin{rem}\label{Rem:Alg_Ind}
 In \cite{CY07}, Chang and Yu proved that the following elements are algebraically independent over $\overline{k}$:  
\[
\tilde{\pi}, \zeta_A(n_1), \zeta_A(n_2), \ldots, \zeta_A(n_d)
\]
where $n_1, \ldots, n_d$ are $d$ distinct positive integers such that $n_i$ is not divisible by $q-1$ for each $i$ and $n_i/n_j$ is not an integral power of $p$ for each $i\neq j$. In \cite{Mi15a}, Mishiba showed that $\tilde{\pi}$, $\zeta(n)$, $\zeta(n,n)$ with $2n$ not divisible by $q-1$ are algebraically independent and in \cite{Mi17}, showed that the following $1+d(d+1)/2$ elements are algebraically independent over $\overline{k}$:
\[
    \{\tilde{\pi}\}\cup\{\zeta_A(n_i)\ |\ 1\leq i\leq d \}\cup\{\zeta_A(n_i, n_{i+1})\ |\ 1\leq i\leq d-1\}\cup\cdots\cup\{\zeta_A(n_1, \ldots, n_d)\} 
\]
where $n_1, \ldots, n_d$ satisfy the same condition as that in Chang and Yu's result. Moreover, Mishiba proved a refined version of \cite{Mi17}; we refer the reader to \cite{Mi15b} for details.
\end{rem}

\section*{Acknowledgements}
The authors are deeply grateful to Prof. C.-Y. Chang for directing them to this project and giving them fruitful suggestions. They also gratefully acknowledge Prof. Y. Mishiba for his careful reading and noting of a mistake in an earlier version of this manuscript. They are further thankful to Dr. O. Gezmi\c{s} and Prof. F. Pellarin for providing many valuable comments on this paper. The first author is partially supported by Prof. C.-Y. Chang's MOST Grant 107-2628-M-007-002-MY4. The second author would like to thank JSPS Overseas Research Fellowships and the National Center for Theoretical Sciences in Hsinchu for their support.


\begin{thebibliography}{Utah}
\bibitem[A86]{A86} G. W. Anderson, {\it $t$-motives}, Duke Math. J. {\bf 53} (1986), no. 2, 457--502.

\bibitem[ABP04]{ABP04} G. W. Anderson, W. D. Brownawell and M. A. Papanikolas, {\it Determination of the algebraic relations among special $\Gamma$-values}, Ann. Math. {\bf 160}, no.2, (2004), 237--313.

\bibitem[AT90]{AT90} G.\ W.\ Anderson and D.\ S.\ Thakur, \textit{Tensor powers of the Carlitz module and zeta values}, Ann. Math. (2) \textbf{132} (1990), no. 1, 159--191.

\bibitem[AT09]{AT09} G.\ W.\ Anderson and D.\ S.\ Thakur, \textit{Multizeta values for $\mathbb{F}_q[t]$, their period interpretation, and relations
between them}, Int. Math. Res. Not. (2009), no. 11, 2038--2055.


\bibitem[BK97]{BK97} D. J. Broadhurst and D. Kreimer, {\it Association of multiple zeta values with positive knots via Feynman diagrams up to 9 loops}, Phys. Lett. {\bf B 393} (1997), 403--412.

\bibitem [Ca35]{Ca35}
L.\ Carlitz, \textit{On certain functions connected with polynomials in a Galois field}, Duke Math. J. \textbf{1} (1935), no. 2, 137--168.

\bibitem[C14]{C14}
C.-Y.\ Chang, \textit{Linear independence of monomials of multizeta values in positive characteristic}, Compos. Math. \textbf{150} (2014), 1789--1808.

\bibitem[C16]{C16} C.-Y. Chang, {\it Linear relations among double zeta values in positive characteristic} Cambridge J. Math. {\bf 4} (2016), No.3, 289--331.

\bibitem[CM20]{CM20} C.-Y. Chang and Y. Mishiba, {\it On a conjecture of Furusho over function fields}, to appear in Invent. Math.

\bibitem[CPY19]{CPY19} C.-Y. Chang, M. A. Papanikolas and J. Yu, {\it An effective criterion for Eulerian multizeta values in positive characteristic}, J. Eur. Math. Soc. (2) {\bf 45}, (2019), 405--440. 

\bibitem[CY07]{CY07} C.-Y. Chang and J. Yu, {\it Determination of algebraic relations among special zeta values in positive characteristic}, Adv. Math. {\bf 216} (2007) 321--345.

\bibitem[Ch17]{Ch17} H.-J. Chen, {\it Anderson-Thakur polynomials and multizeta values in positive characteristic}, Asian J. Math. 21 (2017), 1135--1152.


\bibitem[GKZ06]{GKZ06} H. Gangl, M. Kaneko and D. Zagier, {\it Double zeta values and modular forms}, Automorphic forms and Zeta functions, World Scientific Publishing, New York, (2006), 71--106.

\bibitem[Go96]{Go96}
D.\ Goss, {\it Basic structures of function field arithmetic}, Springer-Verlag, Berlin, 1996.

\bibitem[G97]{G97} A. B. Goncharov, {\it The double logarithm and Manin's complex for modular curves }, Math. Res. Lett. {\bf 4} (1997), 627--636.

\bibitem[G98]{G98} A. B. Goncharov, {\it  Multiple polylogarithms, cyclotomy and modular complexes}, Math. Res. Lett. {\bf 5} (1998), no. 4, 497--516. 

\bibitem[GP20]{GP20} O. Gezmi\c{s} and F. Pellarin, {\it Trivial multiple zeta values in Tate algebras}, arXiv:2008.07144, 2020.

\bibitem[H20]{H20} R. Harada, {\it Alternating multizeta values in positive characteristic}, to appear in Math. Z.

\bibitem[HJ16]{HJ16} U. Hartl and A.-K. Juschka, {\it Pink's theory of Hodge structures over function fields}, $t$-motives: Hodge structures, transcendence and other motivic aspects, European Mathematical Society Congress Reports (2020), 31--182. 

\bibitem[HP04]{HP04} U. Hartl and R. Pink, {\it Vector bundles with a Frobenius structure on the punctured unit disc}, Compos.Math. 140 (2004), no. 3, 689--716.


\bibitem[IO08]{IO08} K. Ihara, H. Ochiai, {\it Symmetry on linear relations for multiple zeta values}, Nagoya Math. J. {\bf 189} (2008), 49--62.


\bibitem[LRT14]{LRT14} J. A. Lara Rodriguez and D. Thakur, {\it Zeta-like multizeta values for $\mathbb{F}_q[t]$}, Indian J. Pure Appl. Math. 45 (5), 785-798 (2014).

\bibitem[Mi15a]{Mi15a} Y. Mishiba, {\it Algebraic independence of the Carlitz period and the positive characteristic multizeta values at $n$ and $(n,n)$}, Proc. Amer. Math. Soc. {\bf 143} (2015), 3753--3763.

\bibitem[Mi15b]{Mi15b} Y. Mishiba, {\it $p$-th power relations and Euler-Carlitz relations among multizeta values}, RIMS Kokyuroku Bessatsu (2015), B53: 13--29


\bibitem[Mi17]{Mi17} Y. Mishiba, {\it On algebraic independence of certain multizeta values in characteristic p}, J. Number Theory {\bf 173} (2017), 512--528.

\bibitem[P08]{P08} M. A. Papanikolas, {\it Tannakian duality for Anderson-Drinfeld motives and algebraic independence of Carlitz logarithms}, Invent. Math. {\bf 171}, no.1, (2008), 123--174.

\bibitem[PR03]{PR03} M. A. Papanikolas and N. Ramachandran, {\it A Weil-Barsotti formula for Drinfeld modules}, J. Number Theory 98 (2003), no. 2, 407--431.


\bibitem[Ta10]{Ta10} L. Taelman, {\it 1-t-motifs}, arXiv:0908.1503, 2010.


\bibitem[Th04]{Th04}
D.\ S.\ Thakur, \textit{Function field arithmetic}, World Scientific Publishing, River Edge NJ, 2004.

\bibitem[Th09a]{Th09a}
D.\ S.\ Thakur, \textit{Power sums with applications to multizeta and zeta zero distribution for $\mathbb{F}_q[t]$}, Finite Fields Appl. {\bf 15} (2009), no. 4, 534--552.

\bibitem[Th09b]{Th09b}
D.\ S.\ Thakur, \textit{Relations between multizeta values for $\mathbb{F}_q[t]$}, Int. Math. Res. Notices (2009), no.12, 2318--2346.

\bibitem[Th10]{Th10}
D.\ S.\ Thakur, \textit{Shuffle relations for function field multizeta values}, Int. Math. Res. Notices (2010), no.11, 1973--1980

\bibitem[To18]{To18}
G.\ Todd, \textit{A Conjectural Characterization for $\mathbb{F}_q(t)$-Linear Relations between Multizeta Values}, J. Number Theory \textbf{187}, 264--287 (2018).

\bibitem[W12]{W12} M. Waldschmidt, {\it Lectures on Multiple Zeta Values, IMSC 2011}, available at \url{http://www.math.jussieu.fr/~miw/articles/pdf/MZV2011IMSc.pdf}, 2012.

\bibitem[Z93]{Z93} D. Zagier, {\it Periods of modular forms, traces of Hecke operators, and multiple zeta values}, Research into automorphic forms and L functions (Japanese) (Kyoto, 1992). Surikaisekikenkyusho Kokyuroku No. 843 (1993), 162--170.

\bibitem[Z94]{Z94} D. Zagier, {\it Values of zeta functions and their applications}, in ECM volume, Progress in Math. {\bf 120}, (1994), 497--512.

\bibitem[Zh16]{Zh16} J. Zhao, {\it Multiple zeta functions, multiple polylogarithms and their special values}, Series on Number Theory and its Applications, 12. World Scientific Publishing Co. Pte. Ltd., Hackensack, NJ, (2016).

\end{thebibliography}
\end{document}